\documentclass[oneside,a4paper,11pt,notitlepage]{article}
\usepackage[left=0.9in, right=0.9in, top=1.5in, bottom=1.5in]{geometry}
\usepackage[T1]{fontenc} 
\usepackage[utf8]{inputenc} 
\usepackage[english]{babel} 
\usepackage{lipsum} 
\usepackage{lmodern}
\usepackage{amssymb}
\usepackage{amsthm}
\usepackage{bm}
\usepackage{mathtools}
\usepackage{xcolor}
\usepackage{dsfont}
\usepackage{enumerate}

\usepackage{yhmath}
\usepackage{stmaryrd}
\usepackage{amsmath}
\usepackage{tikz}
\usepackage{verbatim}
\usepackage{graphicx}
\usepackage{enumitem}
\usepackage{footnote} %

\usepackage{braket}
\usepackage{esint}
\newcommand{\abs}[1]{{\left|#1\right|}}
\newcommand{\norma}[1]{{\left\Vert#1\right\Vert}}

\usepackage{booktabs}
\usepackage{graphicx}
\usepackage{tikz}
\usetikzlibrary{patterns}
\usepackage{multicol}
\usepackage{caption}
\usepackage{enumerate}
\usepackage[skins,theorems]{tcolorbox}
\tcbset{highlight math style={enhanced,
		colframe=black,colback=white,arc=0pt,boxrule=1pt}}
\theoremstyle{definition}
\newtheorem{definizione}{Definition}[section]
\theoremstyle{plain}
\newtheorem{teorema}{Theorem}[section]

\newtheorem{lemma}[teorema]{Lemma}
\newtheorem{prop}[teorema]{Proposition}
\newtheorem{corollario}[teorema]{Corollary}
\theoremstyle{definition}
\newtheorem{esempio}{Example}[section]
\newtheorem{oss}[esempio]{Remark}
\newtheorem{open}[esempio]{Open problem}

\renewcommand{\div}{\text{div}}

\DeclareMathOperator{\R}{\mathbb{R}}

\DeclareMathOperator{\Om}{\Omega}

\title{Rigidity results for the $p$-Laplacian Poisson problem with Robin boundary conditions }
\author{Alba Lia Masiello*, Gloria Paoli}
\date{}

\newcommand{\Addresses}{{
\bigskip 
  \footnotesize

  \textit{E-mail address}, A.L.~Masiello (corresponding author): \texttt{albalia.masiello@unina.it} 
   \medskip

 \textsc{Dipartimento di Matematica e Applicazioni ``R. Caccioppoli'', Universit\`a degli studi di Napoli Federico II, Via Cintia, Complesso Universitario Monte S. Angelo, 80126 Napoli, Italy.}
 
  \medskip 
  
   \textit{E-mail address}, G.~Paoli: \texttt{gloria.paoli@fau.de} 
   
 \medskip
 
 \textsc{ Department of Data Science (DDS)
Chair in Dynamics, Control and Numerics (Alexander von Humboldt-Professorship),
Cauerstr. 11,
91058 Erlangen, Germany.}
 \medskip

 \par\nopagebreak 

}}

\begin{document}
\maketitle 
\begin{abstract}
Let $\Omega \subset \mathbb{R}^n$ be an open, bounded and Lipschitz set. We consider the Poisson problem for the $p-$Laplace operator associated to $\Omega$ with Robin boundary conditions. In this setting, we study the equality case in the Talenti-type comparison stated in \cite{AGM}. We prove that the equality is achieved only if $\Omega$ is a ball and both the function $u$ and the right hand side $f$ of the Poisson equation are radial.
\end{abstract}

\Addresses 

\section{Introduction} 
Symmetrization techniques in the context of qualitative properties of solutions to second-order elliptic boundary value problems are introduced by Talenti in \cite{T}. In this seminal paper, the author considers an open, bounded and Lipschitz set $\Omega\subset\mathbb{R}^n$, the ball $\Omega^\sharp$ with the same measure as $\Omega$ and the solutions $u$ and $v$ to the following problems 
\begin{equation}\label{didi}
    \begin{cases}
    -\Delta u_D=f \, &\text{in } \Om, \\
     u_D=0 & \text{on } \partial\Omega,
    \end{cases} \quad \qquad\quad
    \begin{cases}
    -\Delta v_D=f^\sharp \, & \text{in } \Om^\sharp, \\
   v_D=0 & \text{on } \partial\Omega^\sharp,
    \end{cases}
\end{equation}
where $f\in L^2(\Omega)$ is a positive function and $f^\sharp$ is its Schwarz rearrangement of $f$ (see Definition \ref{rear}). 
In this setting, Talenti proves the following point-wise estimate:
\begin{equation}\label{tal}
    u_D^\sharp(x) \le v_D(x), \quad \text{ for all } x\in\Omega^\sharp.
\end{equation}
For sake of completeness, we observe that this result is proved more generally for a uniformly elliptic linear operator in divergence form. 

A version of this result for nonlinear operators in divergence form is contained in \cite{T2}, which includes as a special instance the case of the $p$-Laplace operator. Moreover, these results are later extended, for instance, to the anisotropic elliptic operators in \cite{AFLT}, to the parabolic case in \cite{ALT}, and to 
higher order operators in \cite{AB,T3}.

Once a comparison result holds, it is natural to ask whether the equality cases can be characterized and, so, if a rigidity result is in force.
In \cite{lions_remark}, the rigidity result linked to problem \eqref{didi} is proved. Indeed, the authors prove that if equality holds in \eqref{tal}, then $\Omega$ is a ball, $u$ is radially symmetric and decreasing and $f=f^\sharp$.
Rigidity results for a generic linear, elliptic second order operator can be found in \cite{posteraro} and \cite{kesavan}. To the best of our knowledge,  rigidity results for nonlinear operators with Dirichlet boundary conditions are not present in the literature. In this paper, we obtain, as a corollary of our results, the rigidity for the $p-$Laplace operator with Dirichlet boundary conditions in any dimension (see Corollary \ref{cordir}). 


For a long time, it was believed that comparison results could not be proved by means of spherical rearrangement argument when dealing with Robin boundary conditions, until the recent paper \cite{ANT}. The 
authors consider the following problems
\begin{equation*}
    \begin{cases}
    -\Delta u=f \, &\text{in } \Om, \\[1ex]
    \displaystyle{\frac{\partial u}{\partial\nu}+\beta u=0} & \text{on } \partial\Omega,
    \end{cases} \quad \qquad\quad
    \begin{cases}
    -\Delta v=f^\sharp \, & \text{in } \Om^\sharp, \\[1ex]
    \displaystyle{\frac{\partial v}{\partial\nu}+\beta v=0} & \text{on } \partial\Omega^\sharp,
    \end{cases}
\end{equation*}
and they prove a comparison result involving Lorentz norms of $u$ and $v$ whenever $f$ is a non negative function in $L^2(\Omega)$ and $\beta$ is a positive parameter. In particular, in the case $f\equiv 1$, they prove
\begin{equation*}
    \norma{u}_{L^p(\Omega)}\le \norma{v}_{L^p(\Omega^\sharp)}, \quad p=1,2,
\end{equation*}
 and, if $n=2$, the pointwise comparison holds:
\begin{equation}\label{trombetti}
    u^\sharp(x) \le v(x), \quad \text{ for all } x\in\Omega^\sharp.
\end{equation}

In \cite{NOI}, it is proved that \eqref{trombetti} is rigid, i.e. the equality case is possible if and only if $\Omega$ is a ball and $u$ is a radial and decreasing function.

Generalizations of the results contained in \cite{ANT} can be found for the anisotropic case in \cite{San2}, for mixed boundary conditions in \cite{ACNT}, in the case of the Hermite operator in \cite{nunzia2022sharp}. 

In the present paper, we focus our study on the rigidity for $p-$Laplace operator. In this case, the comparison results are obtained in \cite{AGM} and the setting is the following.

Let $\Omega$ be a bounded, open and Lipschitz set in $\R^n$, with $n\ge 2$. 
 Let $p\in(1,+\infty)$ and let $f\in L^{p'}(\Omega)$ be a positive function, where $p'=p/(p-1)$. The Poisson problem for the $p-$Laplace operator with Robin boundary conditions is 
\begin{equation}
\label{rob1}
\begin{cases}
-\Delta_p u:= -\div(\abs{\nabla u}^{p-2} \nabla u)=f & \text{ in } \Omega \\[1ex]
\abs{\nabla u}^{p-2} \displaystyle{\frac{\partial u}{\partial \nu}} + \beta  \abs{u}^{p-2}u =0  & \text{ on } \partial \Omega,
\end{cases}
\end{equation}
where $\nu$ is the unit exterior normal to $\partial\Omega$ and $\beta>0$.
A function $u \in W^{1,p}(\Omega)$ is a weak solution to \eqref{rob1} if
\begin{equation}
\label{weak-f}
\int_{\Omega} \abs{\nabla u}^{p-2}\nabla u \nabla \varphi \, dx + \beta \int_{\partial \Omega} \abs{u}^{p-2} u \varphi \, d\mathcal{H}^{n-1}(x) = \int_{\Omega} f \varphi \, dx, \qquad \forall \varphi \in W^{1,p}(\Omega).
\end{equation}

The symmetrized problem associated to \eqref{rob1} is the following
\begin{equation}
\label{rob2}
\begin{cases}
-\Delta_p v= f^{\sharp} & \text{ in } \Omega^\sharp \\[1ex]
\abs{\nabla v}^{p-2}  \displaystyle{\frac{\partial v}{\partial \nu}} + \beta \abs{v}^{p-2}  v =0  & \text{ on } \partial \Omega^\sharp.
\end{cases}
\end{equation}
In \cite{AGM} the authors establish a comparison result between suitable Lorentz norms (see Definition \ref{lorentz}) of the solutions $u$ and $v$ to problems \eqref{rob1} and \eqref{rob2} respectively. In particular, they prove
	\begin{equation}
	    \label{fgen2}
	    \norma{u}_{L^{pk,p}(\Omega)} \, \leq \norma{v}_{L^{pk,p}(\Omega^\sharp)}, \quad \, \; \forall \, 0 < k \leq \frac{n(p-1)}{(n-2)p +n},
	\end{equation}
and in the case $f\equiv 1$, they prove
\begin{equation}
	    \label{f1}
	   u^\sharp(x)\le v(x), \quad 1 \leq p \leq \frac{n}{n-1}
	\end{equation}
and

\begin{equation}\label{altra}
	    \norma{u}_{L^{pk,p}(\Omega)} \, \leq \norma{v}_{L^{pk,p}(\Omega^\sharp)}, \quad \, \; \forall \, 0 < k \leq \frac{n(p-1)}{(n-2)p +n}, \quad \forall p>1.
\end{equation}

In the present paper, we want to characterize the equality case in \eqref{fgen2} and \eqref{altra}, answering to the open problem contained in \cite{NOI}.  
For simplicity, we state the main Theorem only in the case $f\in L^{p'}(\Omega)$ positive, since in the case $f\equiv 1$ the proof is analogous, as we observe in Remark \ref{osser}.

\begin{teorema}\label{th1}
Let $\Omega\subset \mathbb{R}^n$ be a bounded, open and Lipschitz set and let $\Omega^\sharp$ be the ball centered at the origin with the same measure as $\Omega$. 
Let $u$ be the solution to \eqref{rob1} and let $v$ be a solution to \eqref{rob2}. If  
\begin{equation}\label{nolo}
    \norma{u}_{L^{pk,p}(\Omega)}=\norma{v}_{L^{pk,p}(\Omega^\sharp)}, \quad \text{for some }  k\in \left]0, \frac{n(p-1)}{(n-2)p+n}\right]
\end{equation}
  then,  there exists $x_0\in \R^n$ such that
\begin{equation*}
    \Omega=\Omega^\sharp +x_0, \qquad u(\cdot+x_0)=v(\cdot), \qquad f(\cdot+ x_0)=f^\sharp(\cdot).
\end{equation*}
\end{teorema}

The idea of the proof of Theorem \ref{th1} is the following. First of all, we prove that hypothesis \eqref{nolo} implies that the superlevel sets of $u$ are balls. The main difficulty is to prove that these balls are concentric. 

Differently from the case of the Laplace operator with Dirichlet boundary conditions studied in \cite{ALT,FerVol}, we can't apply directly the steepest descent method introduced in \cite{AroTal}, because it strongly relays on the continuity of both the solution and of its gradient. In the case of the $p-$Laplace equation, the continuity of the solution up to the boundary depends on the regularity of the given datum $f$. 
To overcome this regularity issue we show that $u$ is a solution to a suitable Dirichlet problem and it satisfies the P\'olya-Szegő inequality with equality sign.  
Then, we can conclude that $u$ is radially symmetric and decreasing, using the classical result contained in \cite{Brothers1988}.
We make use of Lemma \ref{dirichlet_lemma}, where the rigidity of the Poisson problem for the $p$-Laplace operator with Dirichlet boundary condition is proved under the assumption $f\in L^{p'}(\Omega)$ and positive.
Up to our knowledge, Corollary  \ref{cordir} seems to be new in the literature.

\vspace{4mm}
The paper is organized as follows. In Section $2$ we recall some definitions about rearrangement of functions and we state some lemmas that we will need in the proof of the main theorem. Section $3$ is dedicated to the proof of the main result and we conclude with a list of open problems. 

\section{Notation and Preliminaries}
Throughout this article we will denote by $|\Omega|$ the Lebesgue measure of an open and bounded Lipschitz set of $\mathbb{R}^n$, with $n\geq 2$, and by $P(\Omega)$ the perimeter of $\Omega$. Since we are assuming that $\partial \Omega$ is Lipschitz, we have that $P(\Omega)=\mathcal{H}^{n-1}(\partial\Omega)$, where $\mathcal{H}^{n-1}$ denotes the $(n-1)-$dimensional Hausdorff measure.

We recall the classical isoperimetric inequality and and we refer the reader, for example, to \cite{ossy,burago,chavel,talenti} and to the original paper by De Giorgi \cite{degiorgi}.
\begin{teorema}[Isoperimetric Inequality]
Let $E\subset \R^n$ be a set of finite perimeter. Then,
\begin{equation}
   \label{isoperimetrica}
    n \omega_n^{\frac{1}{n}} \abs{E}^{\frac{n-1}{n}}\le P(E),
\end{equation}
where $\omega_n$ is the measure of the unit ball in $\R^n$.
Equality occurs if and only if $E$ is (equivalent to) a Ball.
\end{teorema}
For the following theorem, we refer to \cite{ambrosio2000functions}.
 \begin{teorema}[Coarea formula]
 Let $\Omega \subset \mathbb{R}^n$ be an open set with Lipschitz boundary. Let $f\in W^{1,1}_{\text{loc}}(\Omega)$ and let $u:\Omega\to\R$ be a measurable function. Then,
 \begin{equation}
   \label{coarea}
   {\displaystyle \int _{\Omega}u(x)|\nabla f(x)|dx=\int _{\mathbb {R} }dt\int_{\Omega\cap f^{-1}(t)}u(y)\, d\mathcal {H}^{n-1}(y)}.
 \end{equation}
 \end{teorema}
Let us recall some basic notions about rearrangements. For a general overview, see, for instance, \cite{kes}.

 \begin{definizione}
	Let $u: \Omega \to \R$ be a measurable function, the \emph{distribution function} of $u$ is the function $\mu : [0,+\infty[\, \to [0, +\infty[$ defined as the measure of the superlevel sets of $u$, i.e.
	$$
	\mu(t)= \abs{\Set{x \in \Omega \, :\,  \abs{u(x)} > t}}.
	$$
\end{definizione}
\begin{definizione} 
	Let $u: \Omega \to \R$ be a measurable function, the \emph{decreasing rearrangement} of $u$ is the distribution function of $\mu $. We will denote it by $u^\ast(\cdot)$.
	\end{definizione}
	\begin{oss}\label{inverse}
	Let us notice that the function $\mu(\cdot)$ is decreasing and right continuous and the function $u^\ast(\cdot)$ is its generalized inverse.
	\end{oss}

	\begin{definizione}\label{rear}
	 The \emph{Schwartz rearrangement} of $u$ is the function $u^\sharp $ whose superlevel sets are balls with the same measure as the superlevel sets of $u$. 
	\end{definizione}
		We have the following relation between $u^\sharp$ and $u^*$:
	$$u^\sharp (x)= u^*(\omega_n\abs{x}^n),$$
 where $\omega_n$ is the measure of the unit ball in $\R^n$, and one can easily check that the functions $u$, $u^*$ e $u^\sharp$ are equi-distributed, i.e. they have the same distribution function, and it holds
$$ \displaystyle{\norma{u}_{L^p(\Omega)}=\norma{u^*}_{L^p(0, \abs{\Omega})}=\lVert{u^\sharp}\rVert_{L^p(\Omega^\sharp)}}, \quad \text{for all } p\ge1.$$

We also recall the Hardy-Littlewood inequality, an important propriety of the decreasing rearrangement,
\begin{equation*}
 \int_{\Omega} \abs{h(x)g(x)} \, dx \le \int_{0}^{\abs{\Omega}} h^*(s) g^*(s) \, ds.
\end{equation*}

\noindent So, by choosing $h(\cdot)=\chi_{\left\lbrace\abs{u}>t\right\rbrace}$, one has
\begin{equation*}
\int_{\abs{u}>t} \abs{g(x)} \, dx \le \int_{0}^{\mu(t)} g^*(s) \, ds.
\end{equation*}
We now introduce the Lorentz spaces (see \cite{T3} for more details on this topic). 
\begin{definizione}\label{lorentz}
Let $0<p<+\infty$ and $0<q\le +\infty$. The Lorentz space $L^{p,q}(\Omega)$ is the space of those functions such that the quantity:
\begin{equation*}
    \norma{u}_{L^{p,q}} =
    \begin{cases}
   	 \displaystyle{ p^{\frac{1}{q}} \left( \int_{0}^{\infty}  t^{q} \mu(t)^{\frac{q}{p}}\, \frac{dt}{t}\right)^{\frac{1}{q}}} & 0<q<\infty\\[2ex]
	 \displaystyle{\sup_{t>0} \, (t^p \mu(t))} & q=\infty
	\end{cases}
\end{equation*}
is finite.
\end{definizione}
	
 Let us observe that for $p=q$ the Lorentz space coincides with the $L^p$ space, as a consequence of the \emph{Cavalieri's Principle}

\[
\int_\Omega \abs{u}^p =p \int_0^{+\infty} t^{p-1} \mu(t) \, dt.
\]
The solutions $u$ to problem \eqref{rob1} and $v$ to problem \eqref{rob2} are both $p$-superharmonic and, as a consequence of the strong maximum principle and the lower semicontinuity (see \cite{V,linp}), they achieve their minima on the boundary.
If we set 
$$u_m=\min_\Omega u, \quad v_m=\min_{\Omega^\sharp} v$$ 
the positiveness of $\beta$ and the Robin boundary conditions leads to $u_m \geq 0$ and $v_m \geq 0$. Hence, $u$ and $v$ are strictly positive in the interior of $\Omega$.
Moreover, we can observe that 
\begin{equation}
		\label{minima_eq}
		u_m = \min_\Omega u \leq  \min_{\Omega^\sharp} v= v_m,
	\end{equation}
indeed,

	\begin{equation*}
		\begin{split}
			 v_m^{p-1}  \text{P}(\Omega^\sharp) &= \int_{\partial \Omega^\sharp} v(x)^{p-1} \, d\mathcal{H}^{n-1}(x)= \frac{1}{\beta}\int_{\Omega^\sharp} f^\sharp \, dx=\frac{1}{\beta} \int_{\Omega} f \, dx \\
			& = \int_{\partial \Omega} u(x)^{p-1} \, d\mathcal{H}^{n-1}(x) \\
			&\geq u_m^{p-1}  \text{P}(\Omega)  \geq { u_m}^{p-1} \text{P}(\Omega^\sharp).
		\end{split}
	\end{equation*}

	Moreover, it holds
	\begin{equation}
		\label{mf}
		\mu (t) \leq \phi (t) = \abs{\Omega}, \quad \forall t \leq v_m.
	\end{equation}

Now, for $t\geq 0$, we introduce the following notations:
$$U_t=\left\lbrace x\in \Omega : u(x)>t\right\rbrace \quad \partial U_t^{int}=\partial U_t \cap \Omega, \quad \partial U_t^{ext}=\partial U_t \cap \partial\Omega, \quad \mu(t)=\abs{U_t}$$
and
$$V_t=\left\lbrace x\in \Omega^\sharp : v(x)> t\right\rbrace, \quad \partial V_t^{int}=\partial V_t \cap \Omega, \quad \partial V_t^{ext}=\partial V_t \cap \partial\Omega, \quad \phi(t)=\abs{V_t}.$$

Because of the invariance of the $p-$Laplacian and of the Schwarz rearrangement of $f$ by rotation, the solution $v$ to \eqref{rob2} is radial, so the set $V_t$ are balls.

Now, we recall some technical Lemmas, proved in \cite{AGM}, that we need in what follows. We recall the proof of Lemma \ref{key} for reader's convenience, while we omit the proof of Lemma \ref{lemma3.3} and Lemma \ref{lem_Gronwall}.

\begin{lemma}
\label{key}
Let $u$ be the solution to \eqref{rob1} and let $v$ be the solution to \eqref{rob2}. Then, for almost every $t >0$, we have
	\begin{equation} 
	\label{talentimu}
\gamma_n \mu(t)^{\left(1-\frac{1}{n}\right)\frac{p}{p-1}} \leq \left(\int_0^{\mu (t)}f^\ast (s ) \, ds\right)^{\frac{1}{p-1}} \left( - \mu'(t) + \frac{1}{\beta ^{\frac{1}{p-1}}}\int_{\partial U_t^\text{ext}} \frac{1}{u} \, d\mathcal{H}^{n-1}(x)\right)
\end{equation}
	and
\begin{equation} 
	\label{talentiphi}
	\gamma_n \phi(t)^{\left(1-\frac{1}{n}\right)\frac{p}{p-1}} = \left(\int_0^{\phi (t)}f^\ast (s ) \, ds\right)^{\frac{1}{p-1}} \left( - \phi'(t) + \frac{1}{\beta ^{\frac{1}{p-1}}}\int_{\partial V_t^\text{ext}} \frac{1}{v} \, d\mathcal{H}^{n-1}(x)\right).
	\end{equation}
		where $\gamma_n= \left(n \omega_n^{1/n}\right)^{\frac{p}{p-1}}$.
\end{lemma}
\begin{proof}
 Let $t >0$ and $h >0$. In the weak formulation \eqref{weak-f}, we choose the following test function 
	\begin{equation}\label{phi}
	\left.
	\varphi (x)= 
	\right.
	\begin{cases}
	0 & \text{ if }u < t \\
	u-t & \text{ if }t< u < t+h \\
	h & \text{ if }u > t+h,
	\end{cases} 
	\end{equation}
obtaining
	
	\begin{equation}\label{weaky}
	\begin{split}
	\int_{U_t \setminus U_{t+h}} \abs{\nabla u}^p\, dx &+ \beta h \int_{\partial U_{t+h}^{ ext }} u^{p-1}  \, d\mathcal{H}^{n-1}(x) + \beta \int_{\partial U_{t}^{ ext }\setminus \partial U_{t+h}^{ ext }} u^{p-1} (u-t) \, d\mathcal{H}^{n-1}(x) \\
	& =  \int_{U_t \setminus U_{t+h}} f (u-t ) \, dx + h \int_{U_{t+h}} f \, dx.
	\end{split}
	\end{equation}
		Dividing \eqref{weaky} by $h$, using coarea formula and letting $h$ go to $0$, we have that for a.e. $t>0$
	\begin{equation*}
	\int_{\partial U_t} g(x) \, d\mathcal{H}^{n-1} = \int_{U_{t}} f\, dx ,
	\end{equation*}
	where
	\begin{equation}\label{fung}
	\begin{cases}
	    	\abs{\nabla u }^{p-1}& \text{ if }x \in \partial U_t^{ int },\\
	\beta u ^{p-1}& \text{ if }x \in \partial U_t^{ ext }.
	\end{cases}
	\end{equation}
Using the isoperimetric inequality, for a.e. $t\in [0, +\infty)$ we have 
	\begin{align}
	\label{isop}
	\hspace{-0.8em}    n \omega_n^{\frac{1}{n}}  \mu(t)^{\frac{n-1}{n}} &\leq P(U_t) = \int_{\partial U_t} \,  d\mathcal{H}^{n-1}\\
	    \label{holder}
	    &\leq\left(\int_{\partial U_t}g\, d\mathcal{H}^{n-1}(x)\right)^{\frac{1}{p}} \left(\int_{\partial U_t}\frac{1}{g^{\frac{1}{p-1}}} \, d\mathcal{H}^{n-1}(x)\right)^{1-\frac{1}{p}} \\
	&= \left(\int_{\partial U_t}g \, d\mathcal{H}^{n-1}(x)\right)^{\frac{1}{p}} \left( \int_{\partial U_t^{ int }}\frac{1}{\abs{\nabla u}}\, d\mathcal{H}^{n-1}(x) +\frac{1}{\beta^{\frac{1}{p-1}}} \int_{\partial U_t^{ ext }}\frac{1}{u} \,  d\mathcal{H}^{n-1}(x) \right)^{1-\frac{1}{p}}\\
	\label{hardy}
	&\leq \left(\int_0^{\mu(t)} f^\ast (s) \, ds\right)^{\frac{1}{p}} \left( -\mu'(t) +\frac{1}{\beta^{\frac{1}{p-1}}} \int_{\partial U_t^{ ext }}\frac{1}{u} \,  d\mathcal{H}^{n-1}(x) \right)^{1-\frac{1}{p}},
	\end{align}
	and, so, \eqref{talentimu} follows. Finally, we notice that, if $v$ is the solution to \eqref{rob2}, then all the inequalities above are equalities, and, consequently, we have \eqref{talentiphi}.
\end{proof}

\begin{lemma}
	\label{lemma3.3}
	For all $\tau \geq v_m $, we have
	\begin{equation}
	\label{intmu}
	\int_0^\tau t^{p-1} \left(\int_{\partial U_t^{ ext } } \frac{1}{ u(x) } \, d \mathcal{H}^{n-1}(x)\right) \, dt \leq \frac{1}{p\beta} \int_0^{\abs{\Omega}} f^\ast(s) \,ds.
	\end{equation}
	Moreover,
	\begin{equation}
	\label{intfi}
	\int_0^\tau t^{p-1} \left(\int_{\partial V_t \cap \partial\Omega^\sharp } \frac{1}{ v(x) } \, d \mathcal{H}^{n-1}(x)\right) \, dt =  \frac{1}{p\beta} \int_0^{\abs{\Omega}} f^\ast(s) \,ds,
	\end{equation}
\end{lemma}

\begin{lemma}[Gronwall]\label{lem_Gronwall}
Let $\xi(\tau)$ be a continuously differentiable function, let $q>1$ and let $C$ be a non negative constant $C$ such that the following differential inequality holds 
$$
	\tau \xi' (\tau ) \leq (q-1) \xi(\tau) + C \quad \forall \tau \geq \tau_0 >0.
	$$
Then, we have

\begin{equation}\label{gron1}
    \xi(\tau) \leq \left(\xi(\tau_0) + \frac{C}{q-1}\right) \left( \frac{\tau}{\tau_0}\right)^{q-1} - \frac{C}{q-1}   \quad \forall \tau \geq \tau_0,
\end{equation}
and 
\begin{equation}\label{gron2}
\xi'(\tau) \leq \left( \frac{(q-1)\xi(\tau_0 )+ C}{\tau_0}\right) \left( \frac{\tau}{\tau_0}\right)^{q-2} \quad \forall \tau \geq \tau_0.
\end{equation}
\end{lemma}

The following Lemma is contained in \cite{ALT}.

\begin{lemma}\label{alvino}
Let $f, g\in L^2(\Omega)$ be two positive functions. If 
\begin{equation}
    \int_{\Omega} fg\;dx=\int_{\Omega^\sharp} f^\sharp g^\sharp\;dx,
\end{equation}
    then, for every $\tau\geq 0$ there exists $t\geq 0$ such that we have, up to zero measure set,
    \begin{equation}
        \{ g>\tau \}=\{ f>t\}.    \end{equation}
\end{lemma}
We conclude this preliminary session, recalling the classical results contained in \cite{Brothers1988} (see Theorem $1.1$ and Lemma $2.3$). In particular, the result contained in \ref{3} of Lemma \eqref{brothers} gives the rigidity of the P\'olya-Szegő inequality (see \cite{pol}):
\begin{equation} \label{polya}
    \int_{\R^n} |\nabla u^\sharp|^p \, dx\le \int_{\R^n} \abs{\nabla u}^p  \, dx, \quad \forall u\in W^{1,p}(\mathbb{R}^n).
\end{equation}
\begin{teorema}\label{brothers}
Let $w\in W^{1,p}(\mathbb{R}^n)$, let $\sigma(t)$ be its distribution function 
and let 
$$w_M:=\begin{cases}
\norma{w}_{\infty} & \text{if } w\in L^\infty(\Omega)\\
+\infty & \text{otherwise.}
\end{cases}$$ 
Then, the following are true:
\begin{enumerate}[label=\roman{*}., ref=(\roman{*})]   
    \item \label{1}For almost all $t\in (0,w_M)$,
    \begin{equation}\label{brothers1}
        \infty>-\sigma'(t)\geq\displaystyle \int_{w^{-1}(t)}\dfrac{1}{|\nabla w|}d\mathcal{H}^{n-1}
    \end{equation}
     \item  \label{2} $\sigma$ is absolutely continuous if and only if 
     \begin{equation}\label{brothers2}
        \abs{\Set{\lvert\nabla w^\sharp\rvert =0}\cap\Set{0<w^\sharp<w_M}}=0.
     \end{equation}
     \item \label{3} If 
     \begin{equation}\label{hp_brothers3}
         \int_{\mathbb{R}^n} |\nabla w|^p= \int_{\mathbb{R}^n} |\nabla w^\sharp|^p,
     \end{equation}
     and \eqref{brothers2} holds, then there exist a translate of $w^\sharp$ which is almost everywhere equal to $w$.
\end{enumerate}
\end{teorema}

\begin{oss} \label{remfus} We observe that in \cite{fusco}, it is proved that the condition
  \begin{equation}\label{bro}
        \abs{\Set{\abs{\nabla w} =0}\cap\Set{0<w<w_M}}=0
     \end{equation}
     implies \eqref{brothers2}. 
So, by \ref{3} in Lemma \ref{brothers}, if we have \eqref{hp_brothers3} and \eqref{bro}, there exists a translated of $w^\sharp$ which is almost everywhere equal to $w$. 
\end{oss}

\begin{oss} We observe that the P\'olya -Szegő inequality \eqref{polya} and the relative rigidity result \ref{3} contained in Lemma \ref{brothers} hold also if we assume $w\in W^{1,p}_0(\Omega)$. Indeed, it is easily proved that for every $w\in W^{1,p}_0(\Omega)$ one has $w^\sharp\in W^{1,p}_0(\Omega^\sharp)$.
\end{oss}

\section{Proof of Theorem \ref{th1}}
In order to prove the main Theorem \ref{th1}, we divide the proof into the following steps. First of all, we prove that, under the assumptions of Theorem \ref{th1}, equality holds in \eqref{talentimu} and this is the content of Proposition \ref{thnorme}. Then, in Proposition \ref{teorema1}, we prove that equality in \eqref{talentimu} implies the fact that $\Omega$ is a ball and $u$ and $f$ are radial functions.
In order to prove this last step, 
we need the key Lemma \ref{dirichlet_lemma}.

\begin{prop}
\label{thnorme}
Let $u$ be the solution to \eqref{rob1} and let $v$ be the solution to \eqref{rob2}. If there exists $k$ 
$$ k\in \left]0, \frac{n(p-1)}{(n-2)p+n}\right] \quad \text{such that } \quad \norma{u}_{L^{pk,p}(\Omega)}=\norma{v}_{L^{pk,p}(\Omega^\sharp)},$$
then equality holds in \eqref{talentimu} for almost every $t$.
\end{prop} 

\begin{proof}
Since we are assuming that $\norma{u}_{L^{pk,p}(\Omega)}=\norma{v}_{L^{pk,p}(\Omega^\sharp)}$, we have that
\begin{equation}\label{uguaglianza_norme}
    \int_0^{+\infty} t^{p-1}\mu(t)^{\frac{1}{k}}\, dt=\int_0^{+\infty} t^{p-1}\phi(t)^{\frac{1}{k}}\, dt.
\end{equation}
Let us multiply \eqref{talentimu} by $t^{p-1}\mu(t)^\alpha$, where $ \displaystyle \alpha={\frac{1}{k}-\left(1-\frac{1}{n}\right)\frac{p}{p-1}}$, and let us integrate from 0 to $+\infty$: 

\begin{equation}\label{talenti_mu2}
\begin{aligned}
    &\gamma_n \int_0^{+\infty} t^{p-1}\mu^{\frac{1}{k}}(t)\, dt  \\
    \le&\int_0^{+\infty}  \left(\int_0^{\mu (t)}f^\ast (s ) \, ds\right)^{\frac{1}{p-1}} \left( - \mu'(t) + \frac{1}{\beta ^{\frac{1}{p-1}}}\int_{\partial U_t^{ext}} \frac{1}{u} \, d\mathcal{H}^{n-1}\right)t^{p-1} \mu(t)^{\alpha}\, dt\\
    \le & \int_0^{+\infty} t^{p-1} \mu(t)^{\alpha} \left(\int_0^{\mu (t)}f^\ast (s ) \, ds\right)^{\frac{1}{p-1}}(-\mu'(t))\, dt+ \frac{\abs{\Omega}^\alpha}{p\beta ^{\frac{p}{p-1}}} \left(\int_0^{\abs{\Omega}}f^\ast (s ) \, ds\right)^{\frac{p}{p-1}},
\end{aligned}
\end{equation}

where in the last inequality we have used $\mu(t)\le \abs{\Omega}$ and \eqref{intmu} in Lemma \ref{lemma3.3}. As far as $v$ is concerned, it holds
\begin{equation}\label{talentifi2}
    \begin{aligned}
    &\gamma_n \int_0^{+\infty} t^{p-1}\phi^{\frac{1}{k}}(t)\, dt\\ &= \int_0^{+\infty}  \left(\int_0^{\phi (t)}f^\ast (s ) \, ds\right)^{\frac{1}{p-1}} \left( - \phi'(t) + \frac{1}{\beta ^{\frac{1}{p-1}}}\int_{\partial V_t^{ext}} \frac{1}{u} \, d\mathcal{H}^{n-1}\right)t^{p-1} \phi(t)^{\alpha}\, dt\\
    &= \int_0^{+\infty} t^{p-1} \phi(t)^{\alpha} \left(\int_0^{\phi (t)}f^\ast (s ) \, ds\right)^{\frac{1}{p-1}}(-\phi'(t))\, dt+ \frac{\abs{\Omega}^\alpha}{p\beta ^{\frac{p}{p-1}}} \left(\int_0^{\abs{\Omega}}f^\ast (s ) \, ds\right)^{\frac{p}{p-1}}.
    \end{aligned}
\end{equation}
We observe that the left-hand-side of \eqref{talenti_mu2} and the left-hand-side of \eqref{talentifi2} are equal from  \eqref{uguaglianza_norme}. So, it follows
\begin{equation}\label{prima_dis}
\begin{aligned}
\hspace{-0.7em}\int_0^{+\infty} \!\!\! \!\!\! t^{p-1} \phi(t)^{\alpha} \left(\int_0^{\phi(t)}f^\ast (s ) \, ds\right)^{\frac{1}{p-1}} \!\!\! \!\!\! (-\phi'(t))\, dt
    \le &\int_0^{+\infty} \!\!\! \!\!\! t^{p-1} \mu(t)^{\alpha} \left(\int_0^{\mu (t)}f^\ast (s ) \, ds\right)^{\frac{1}{p-1}} \!\!\!\!\! (-\mu'(t))\, dt.
\end{aligned}
\end{equation}
Setting $\displaystyle{F(l)= \int_0^l \omega^\delta \Biggl( \int_0^\omega f^\ast (s) \, ds \Biggr)^{\frac{1}{p-1}} \! \!\! d\omega}$, and integrating \eqref{prima_dis} by parts, we get

\begin{equation*}
      \int_0^\infty t^{p-2} F(\phi(t)) \, dt \leq \int_0^\infty t^{p-2} F(\mu(t)) \, dt,
\end{equation*}
being $\mu(t)=\phi(t)=0$ for $t>v_M$. In \cite{AGM} (see the proof of Theorem $1.1$), it is proved that 
\begin{equation}\label{ineq_desired1}
        \int_0^\infty t^{p-2} F(\mu(t)) \, dt \leq \int_0^\infty t^{p-2} F(\phi(t)) \, dt.
\end{equation}
and we recall here the proof for the reader's convenience.
In order to do that, we multiply \eqref{talentimu} by $\displaystyle{t^{p-1}F(\mu(t)) \mu(t)^{-\frac{(n-1)p}{n(p-1)} }}$ and we integrate between $0$ and $\tau>v_m$. First, we observe that, by the hypothesis $\displaystyle{k \leq \frac{n(p-1)}{(n-2)p+n}}$, it follows that the function $\displaystyle{h(l)= F(l)l^{-\frac{(n-1)p}{n(p-1)}}}$ is non decreasing. Hence, we obtain
\begin{equation*}
	\begin{split}
		\int_0^\tau \gamma_n t^{p-1}F(\mu(t)) \, dt &\leq  	\int_0^\tau \Bigl(- \mu'(t) \Bigr) t^{p-1}\mu(t)^{-\frac{(n-1)p}{n(p-1)} } F(\mu(t)) \left(\int_0^{\mu (t)}f^\ast (s ) \, ds\right)^{\frac{1}{p-1}} \, dt \\
		&+F(\abs{\Omega})\frac{\abs{\Omega}^{-\frac{p(n-1)}{n(p-1)} }}{p\beta ^{\frac{p}{p-1}}} \left(\int_0^{\abs{\Omega}}f^\ast (s ) \, ds\right)^{\frac{p}{p-1}}.	
	\end{split} 
\end{equation*}

\noindent If we integrate by parts both sides of the last expression and we set
$$\displaystyle{C=F(\abs{\Omega})\frac{\abs{\Omega}^{-\frac{p(n-1)}{n(p-1)} }}{p\beta ^{\frac{p}{p-1}}} \left(\int_0^{\abs{\Omega}}f^\ast (s ) \, ds\right)^{\frac{p}{p-1}}},$$
\noindent we obtain
\begin{equation}
\label{Solidus_Snake}
\tau \int_0^{\tau} \gamma_n t^{p-2} F (\mu(t)) \, dt + \tau H_\mu(\tau) \leq \int_{0}^{\tau} \int_0^t r^{p-2} F(\mu(r)) \, dr dt+ \int_0^{\tau} H_\mu(t) \, dt +C,
\end{equation}
where
\[
H_\mu(\tau)=-\int_{\tau}^{+\infty} t^{p-2} \mu(t)^{-\frac{p(n-1)}{n(p-1)}} F(\mu(t)) \biggl( \int_0^{\mu(t)} f^*(s) \, ds \biggr)^{\frac{1}{p-1}} \, d\mu(t).
\]
Setting now
\begin{equation*}
	\begin{multlined}
		\xi(\tau)=\int_0^\tau \int_0^t \gamma_n r^{p-2}F(\mu(r)) \, dr +\int_0^t H_{\mu}(t) \, dt , 
	\end{multlined} 
\end{equation*}
 inequality \eqref{Solidus_Snake} becomes
$$
\tau \xi'(\tau ) \leq \xi(\tau) +C.
$$
So, Lemma \ref{lem_Gronwall}, with $\tau_0=v_m$ and q=2, gives
$$\int_{0}^{\tau} \gamma_n t^{p-2} F(\mu(t)) \, dt + H_\mu(\tau) \le \left( \frac{\displaystyle{\int_{0}^{v_m} t^{p-2} F(\mu(t) \, dt +H_\mu(v_m) +C}}{v_m}\right).$$

\noindent Of course, the inequality holds as equality if we replace $\mu(t)$ with $\phi(t)$, so we get:
$$
\int_{0}^{\tau} \gamma_n t^{p-2} F(\mu(t)) \, dt + H_\mu(\tau)\le \int_{0}^{\tau} \gamma_n F(\phi(t)) \, dt +H_\phi(\tau),
$$
keeping in mind that $\mu(t)\le \phi(t)= \abs{\Omega} $ for $t \leq v_m$.
Now, letting $\tau \to \infty$, one has
$$
	\int_0^\infty t^{p-2} F(\mu(t)) dt \leq \int_0^\infty t^{p-2} F(\phi(t)) dt,
$$
since
  $H_\mu(\tau), H_\phi(\tau) \to 0$.
  
So, we get equality in \eqref{talenti_mu2} and, consequently, in \eqref{talentimu} for almost every $t$, indeed

\begin{align*}
    &\gamma_n \int_0^{+\infty}  t^{p-1}\mu^{\frac{1}{k}}(t)\, dt \\
    &\le\int_0^{+\infty}  \left(\int_0^{\mu (t)}f^\ast (s ) \, ds\right)^{\frac{1}{p-1}} \left( - \mu'(t) + \frac{1}{\beta ^{\frac{1}{p-1}}}\int_{\partial U_t^{ext}} \frac{1}{u} \, d\mathcal{H}^{n-1}\right)t^{p-1} \mu(t)^{\alpha}\, dt\\[1ex]
    & \le \int_0^{+\infty} t^{p-2} F(\mu(t))\, dt+ \frac{\abs{\Omega}^\alpha}{p\beta ^{\frac{p}{p-1}}} \left(\int_0^{\abs{\Omega}}f^\ast (s ) \, ds\right)^{\frac{p}{p-1}}\\[1ex]
    &=\int_0^{+\infty} t^{p-2} F(\phi(t))\, dt+\frac{\abs{\Omega}^\alpha}{p\beta ^{\frac{p}{p-1}}} \left(\int_0^{\abs{\Omega}}f^\ast (s ) \, ds\right)^{\frac{p}{p-1}}\\[1ex]
    &=\int_0^{+\infty} \left(\int_0^{\phi (t)}f^\ast (s ) \, ds\right)^{\frac{1}{p-1}} \left( - \phi'(t) + \frac{1}{\beta ^{\frac{1}{p-1}}}\int_{\partial V_t^{ext}} \frac{1}{v} \, d\mathcal{H}^{n-1}\right)t^{p-1} \phi(t)^\alpha\, dt\\[1ex]
    &=\gamma_n \int_0^{+\infty}  t^{p-1}\phi^{\frac{1}{k}}(t)\, dt=\gamma_n \int_0^{+\infty}  t^{p-1}\mu^{\frac{1}{k}}(t)\, dt.
\end{align*}

\end{proof}

In the following Lemma 
 we prove that a solution to a Dirichlet problem, such that its distribution function satisfies the differential equation \eqref{talentisigma}, is necessarily defined on a ball and it has to be radial and decreasing.

\begin{lemma}\label{dirichlet_lemma}
Let $\Omega\subset\R^n$ be an open, bounded and Lipschitz set. Let $f\in L^{p'}(\Omega)$ be a positive function, let $w$ be a weak solution to
\begin{equation}
\label{dirichlet}
\begin{cases}
-\Delta_p w= f & \text{ in } \Omega \\
w =0  & \text{ on } \partial \Omega,
\end{cases}
\end{equation}
and let $\sigma$ be the distribution function of $w$. If $\sigma$ satisfies the following condition
\begin{equation} 
	\label{talentisigma}
	\gamma_n \sigma(t)^{\left(1-\frac{1}{n}\right)\frac{p}{p-1}} = \left(\int_0^{\sigma (t)}f^\ast (s ) \, ds\right)^{\frac{1}{p-1}} \left( - \sigma'(t) \right), \quad \text{for a.e. } t\in[0,w_M]
	\end{equation}
 then, there exists $x_0$ such that

$$\Omega=\Omega^\sharp +x_0, \quad w(\cdot +x_0)=w^\sharp(\cdot), \quad f(\cdot+x_0)=f^\sharp(\cdot).$$
\end{lemma}

\begin{proof}
First of all, we recall that $w$ is a weak solution to \eqref{dirichlet} if and only if 
\begin{equation}
    \label{wikdir}
    \int_{\Omega} \abs{\nabla w}^{p-2}\nabla w \nabla \varphi \, dx = \int_{\Omega} f \varphi \, dx, \qquad \forall \varphi \in W_0^{1,p}(\Omega).
\end{equation}
Arguing as in the proof of \eqref{talentimu} in Lemma \ref{key}, choosing the same test function $\varphi$, defined in \eqref{phi},
\begin{equation*}
	\left.
	\varphi (x)= 
	\right.
	\begin{cases}
	0 & \text{ if }w < t \\
	w-t & \text{ if }t< w < t+h \\
	h & \text{ if }w > t+h,
	\end{cases} 
	\end{equation*}
	one obtains
\begin{equation}\label{fmagg}
    \int_{\partial W_t} \abs{\nabla w}^{p-1} \, d\mathcal{H}^{n-1}=\int_{W_t} f(x) \, dx\le \int_0^{\sigma(t)} f^\star(s)\, ds,
\end{equation}
where $W_t=\{x\in\Omega\;:w(x)>t\}$.

If we apply the isoperimetric inequality to the superlevel set $W_t$, the H\"older inequality and the Hardy-Littlewood inequality, we get, for almost every $t$,

\begin{align}
	\label{isop2}
	    n \omega_n^{\frac{1}{n}}  \sigma(t)^{\frac{n-1}{n}}& \leq P(W_t) = \int_{\partial W_t} \,  d\mathcal{H}^{n-1}\\
	    \label{holder2}
	    &\leq\left(\int_{\partial W_t} \abs{\nabla w}^{p-1}\, d\mathcal{H}^{n-1}(x)\right)^{\frac{1}{p}} \left(\int_{\partial W_t}\frac{1}{\abs{\nabla w}} \, d\mathcal{H}^{n-1}(x)\right)^{1-\frac{1}{p}} \\
	\label{hardy2}
        &\leq \left(\int_0^{\sigma(t)} f^\ast (s) \, ds\right)^{\frac{1}{p}} \left( -\sigma'(t) \right)^{1-\frac{1}{p}}.
	\end{align}
So, hypothesis \eqref{talentisigma} ensures us that equality holds in the isoperimetric inequality \eqref{isop2}, in the H\"older inequality \eqref{holder2} and in the Hardy-Littlewood inequality \eqref{hardy2}. \\
We now divide the proof into three steps.\\
\textbf{Step 1}. Let us prove that the superlevel set $\Set{w>t}$ is a ball for all $t\in [0, w_M)$. 
Equality in \eqref{isop2} implies that, for almost every $t$, 
$W_t$ is a ball. On the other hand, for all $t\in [0, w_M)$, there exists a sequence $\Set{t_k}$ such that
\begin{enumerate}
    \item $t_k\to t$;
    \item $t_k>t_{k+1}$;
    \item $\{w>t_k\}$ is a ball for all $k$. 
\end{enumerate}
Since $\displaystyle{\Set{w>t}=\cup_k\Set{w>t_k}}$ can be written as an increasing union of balls, $\{w>t\}$ is a ball for all $t$ and, in particular, $\Om=\{w>0\}$ is a ball too and we obtain that $\Om=x_0+\Om^\sharp$. \\
From now on, we can assume without loss of generality that $x_0=0$.\\
\textbf{Step 2.} Let us prove that the superlevel sets are concentric balls. \\
Equality in \eqref{holder2} implies also equality in H\"older inequality, i.e. 
$$\int_{\partial W_t} \,  d\mathcal{H}^{n-1}
	=\left(\int_{\partial W_t} \abs{\nabla w}^{p-1}\, d\mathcal{H}^{n-1}(x)\right)^{\frac{1}{p}} \left(\int_{\partial W_t}\frac{1}{\abs{\nabla w}} \, d\mathcal{H}^{n-1}(x)\right)^{1-\frac{1}{p}}. $$
 This means that, for almost every $t$, $\abs{\nabla w}$ is constant $\mathcal{H}^{n-1}-$almost everywhere
on $\partial W_t$ , and we denote by $C_t$ the ($\mathcal{H}^{n-1}-$a.e.) constant value of $\abs{\nabla w}$ on $\partial W_t$. We claim that $C_t\neq 0$ for almost every $t$. Indeed, \eqref{fmagg} and the positivity of $f$ ensure us that
$$
    P(W_t) C_t^{p-1}=\int_{\partial W_t} \abs{\nabla w}^{p-1} \, d\mathcal{H}^{n-1}=\int_{W_t} f(x) \, dx>0.$$
Integrating \eqref{talentisigma}, we obtain $w^\sharp(x)=z(x)$, for all $x\in \Omega^\sharp$, where $z$ is the solution to
\begin{equation}
   \begin{cases}
-\Delta_p z= f^\sharp & \text{ in } \Omega^\sharp \\
z =0  & \text{ on } \partial \Omega^\sharp,
\end{cases}
\end{equation}
    and it has the following explicit form:
    
    \begin{equation*}
        z(x)=\int_{\omega_n \abs{x}^n}^{\abs{\Omega}} \frac{1}{\gamma_n} \left(\int_0^s f^\star(r)\, dr\right)^{1/(p-1)}\frac{1}{s^{(1-1/n)(p/(p-1))}} \, ds,
    \end{equation*}
    so it easily follows that
         \begin{equation}\label{pezzotto}
        \abs{\Set{\lvert\nabla w^\sharp\rvert =0}\cap\Set{0<w^\sharp<w_M}}=0.
     \end{equation}
Using \ref{2} in Lemma \ref{brothers}, 
we have that \eqref{pezzotto} 
implies the absolutely continuity of $\sigma$. 

Now, we denote by $C^\sharp_t$ the ($\mathcal{H}^{n-1}-$a.e.) constant value of $\abs{\nabla w^\sharp}$ on $\partial W^\sharp_t$.
We recall that it holds
\begin{equation*}
    -\sigma'(t)= \int_{\partial W^\sharp_t} \frac{1}{\abs{\nabla w^\sharp}}=\frac{P(\partial W^\sharp_t)}{C^\sharp_t}.
\end{equation*}
and, by the absolutely continuity of $\sigma$, we have
\begin{equation*}
  -\sigma'(t) = \int_{\partial W_t} \frac{1}{\abs{\nabla w}}=\frac{P(\partial W_t)}{C_t}.
\end{equation*}
Since $w$ and $w^\sharp$ are equi-distributed, we have,
\begin{equation*}
    \frac{P(\partial W_t)}{C_t}=\frac{P(\partial W^\sharp_t)}{C^\sharp_t}
\end{equation*}
 Moreover, since $P(\partial W_t)=P(\partial W^\sharp_t)$, we have that $C_t=C^\sharp_t$. So, by the coarea formula, we get

\begin{equation*}
\begin{aligned}
\int_\Omega \abs{\nabla w}^p\, dx &=\int_0^{+\infty} \int_{\partial W_t} \abs{\nabla w}^{p-1} \, d\mathcal{H}^{n-1}=\int_0^{+\infty} C_t^{p-1}P(W_t)\, dt \, d\mathcal{H}^{n-1}\\
& =\int_0^{+\infty} \left(C_t^\sharp\right)^{p-1}P(W_t)\, dt \, d\mathcal{H}^{n-1}= \int_0^{+\infty} \int_{\partial W^\sharp_t} \lvert \nabla w^\sharp \rvert^{p-1} \, d\mathcal{H}^{n-1}=\int_{\Omega^\sharp} \lvert \nabla w^\sharp \rvert^p\, dx.
\end{aligned}
\end{equation*}
    By \ref{3} in Lemma \ref{brothers}, we conclude that $u=u^\sharp$.
    
    \noindent
 \textbf{Step 3.} Let us prove that $f$ is radial and decreasing.
 
    \noindent
Equality in \eqref{hardy2} reads, for almost every $t$,
\begin{equation*}
  \int_{W_t} f(x) \, dx =\int_0^{\sigma(t)} f^\ast (s) \, ds.
\end{equation*}
Moreover, for all $\tau\in [0, w_M)$, there exists a sequence $\Set{\tau_k}$ such that
\begin{enumerate}
    \item $\tau_k\to \tau$;
    \item $\tau_k>\tau_{k+1}$;
    \item $\displaystyle\int_{W_{\tau_k}} f(x) \, dx =\displaystyle\int_0^{\sigma(\tau_k)} f^\ast (s) \, ds$, 
\end{enumerate}
and, by the continuity of $\sigma(\cdot)$, we have

\begin{equation*}
    \int_0^{\sigma(\tau)} f^\ast (s) \, ds= \lim_{k}  \int_0^{\sigma(\tau_k)} f^\ast (s)=\lim_k \int_{W_{\tau_k}} f(x) \, dx =\int_{W_\tau} f(x)\;dx.
\end{equation*}
By Lemma \ref{alvino}, we have that for all $\tau$, there exists $\alpha_\tau$ such that 
$$\{w>\tau\}=\{f>\alpha_\tau\}.$$ Consequently, we have that also $f$ is radial and decreasing, so $f=f^\sharp$.
\end{proof}

As a direct consequence of Lemma \ref{dirichlet_lemma}, we obtain the rigidity for the $p-$Laplace operator with Dirichlet boundary conditions. 

\begin{corollario}\label{cordir}
Let $\Omega\subset\R^n$ be an open, bounded and Lipschitz set. Let $f\in L^{p'}(\Omega)$ be a positive function and  let $w$ and $z$ be  weak solutions respectively to
\begin{equation}
        \begin{cases}
-\Delta_p w= f & \text{ in } \Omega \\
w =0  & \text{ on } \partial \Omega,
\end{cases}
\quad 
\begin{cases}
-\Delta_p z= f^{\sharp} & \text{ in } \Omega^\sharp \\
z =0  & \text{ on } \partial \Omega^\sharp.
\end{cases}
\end{equation}
If $w^\sharp(x)=z(x)$, for all $x\in \Omega^\sharp$, then there exists $x_0\in \R^n$ such that
\begin{equation*}
    \Omega=\Omega^\sharp +x_0, \qquad w(\cdot+x_0)=z(\cdot), \qquad f(\cdot+ x_0)=f^\sharp(\cdot).
\end{equation*}
\end{corollario}

\begin{proof}
    From the proof of Lemma \ref{dirichlet_lemma}, it follows that the distribution function of $w$, denoted by $\sigma$, satisfies

    \begin{equation}\label{fata}
        n \omega_n^{\frac{1}{n}}  \sigma(t)^{\frac{n-1}{n}}\leq \left(\int_0^{\sigma(t)} f^\ast (s) \, ds\right)^{\frac{1}{p}} \left( -\sigma'(t) \right)^{1-\frac{1}{p}}.
    \end{equation}
    Now, we integrate \eqref{fata} from $0$ to $t$, obtaining

$$u^\ast(t) =\int_{\sigma(t)}^{\abs{\Omega}} \frac{1}{\gamma_n} \left(\int_0^s f^\star(r)\, dr\right)^{1/(p-1)}\frac{1}{s^{(1-1/n)(p/(p-1))}} \, ds= z^\ast(t).$$
So, if $w^\sharp=z$,  we have $w^\ast=z^\ast$, and consequently we obtain equality in \eqref{fata} for almost every $t\in [0, w_M]$. We can conclude by applying Lemma \ref{dirichlet_lemma}.
    
\end{proof}

Now, using Lemma \ref{dirichlet_lemma}, we are in position to conclude the proof of the main Theorem. 
\begin{prop} \label{teorema1} Let $\Omega\subset \mathbb{R}^n$ be  an open, bounded and Lipschitz set and let $\Omega^\sharp$ be the ball with the same measure as $\Omega$.
Let $u$ be the solution to \eqref{rob1} and let $\mu$ be its distribution function. If equality holds in \eqref{talentimu}, then there exists $x_0\in \R^n$ such that

\begin{equation*}
    \Omega=\Omega^\sharp +x_0, \qquad u(\cdot+x_0)=v(\cdot), \qquad f(\cdot+ x_0)=f^\sharp(\cdot).
\end{equation*}

\end{prop}

\begin{proof}
Firstly, we claim that the superlevel sets $\Set{u>t}$ are balls for every $t\in[0,u_M)$. Equality in \eqref{talentimu} implies the equality in \eqref{isop}, i.e.
$$ n \omega_n^{\frac{1}{n}}  \mu(t)^{\frac{n-1}{n}} = P(U_t), \quad \text{for a. e. } t\in[0,u_M]$$
that means that almost every superlevel set is a ball. Arguing as in Step $1$ of Lemma \ref{dirichlet_lemma}, we can conclude that every superlevel set is a ball, so, $\Om=\{u>u_m\}$ is a ball and we obtain that $\Om=x_0+\Om^\sharp$. 

Let us observe that for every $t,s\in [u_m,u_M]$ with $t<s$, as both $U_t$ and $U_s$ are balls, we have that $\partial U_t\cap\partial U_s$ contains at most one point. In particular, the function $w=u-u_m$ is a weak solution to the Dirichlet problem \eqref{dirichlet} in $\Omega$.

We claim that $\sigma(t)=\abs{\Set{w>t}}$ satisfies \eqref{talentisigma}. Since $\Set{w>t}=\Set{u>t+u_m}$, we have $\sigma(t)=\mu(t+u_m)$ for all $t\in [0, u_M-u_m]$.
Moreover, we have 
$$
\int_{\partial U_t} \frac{1}{u}\, d\mathcal{H}^{n-1}= 0, \quad \forall t>u_m
$$
So, using the fact that we have equality in \eqref{talentimu} by hypothesis, we get 
\begin{align*}
    \gamma_n \sigma(t)^{\left(1-\frac{1}{n}\right)\frac{p}{p-1}} &= \gamma_n \mu(t+u_m)^{\left(1-\frac{1}{n}\right)\frac{p}{p-1}} \\
    & = \left(\int_0^{\mu (t+u_m)}f^\ast (s ) \, ds\right)^{\frac{1}{p-1}} \left( - \mu'(t+u_m) + \frac{1}{\beta ^{\frac{1}{p-1}}}\int_{\partial U_{t+u_m}^\text{ext}} \frac{1}{u} \, d\mathcal{H}^{n-1}(x)\right)\\
    &=\left(\int_0^{\sigma (t)}f^\ast (s ) \, ds\right)^{\frac{1}{p-1}} \left( - \sigma'(t) \right),
	\end{align*}

for all $t\in(0, u_M-u_m)$. So, we can conclude by Lemma \ref{dirichlet_lemma}.
\end{proof}
We conclude now with the proof of the main Theorem. 
\begin{proof}[Proof of Theorem \ref{th1}]
From Proposition \ref{thnorme}, we have that the hypothesis of Theorem \ref{th1}
\begin{equation*}
    \norma{u}_{L^{pk,p}(\Omega)}=\norma{v}_{L^{pk,p}(\Omega^\sharp)}, \quad \text{for some }  k\in \left]0, \frac{n(p-1)}{(n-2)p+n}\right]
\end{equation*}
implies the following equality for almost every $t\in (0,u_M)$
	\begin{equation*} 
\gamma_n \mu(t)^{\left(1-\frac{1}{n}\right)\frac{p}{p-1}} =\left(\int_0^{\mu (t)}f^\ast (s ) \, ds\right)^{\frac{1}{p-1}} \left( - \mu'(t) + \frac{1}{\beta ^{\frac{1}{p-1}}}\int_{\partial U_t^\text{ext}} \frac{1}{u} \, d\mathcal{H}^{n-1}(x)\right), 
\end{equation*}
where $\mu(t)$ is the distribution function of $u$. 

Now, we are in position to apply Proposition \ref{teorema1}, and, so, 
there exists $x_0\in \R^n$ such that
\begin{equation*}
    \Omega=\Omega^\sharp +x_0, \qquad u(\cdot+x_0)=v(\cdot), \qquad f(\cdot+ x_0)=f^\sharp(\cdot).
\end{equation*}

\end{proof}

\section{Remarks and open problems}

\begin{oss}\label{osser}
In \cite{AGM} the authors also prove that in the case $f\equiv 1$, it holds
		\begin{equation}
		\norma{u}_{L^{pk,p}(\Omega)} \leq \norma{v}_{L^{pk,p}(\Omega^{\sharp})}, \quad\quad {\rm if} \:\; \displaystyle{0 <k \leq \frac{n(p-1)}{n(p-1)-p}}.
		\end{equation} 
		We stress that the proof of Theorem \ref{th1} can be adapted to case $f\equiv 1$, regardless of the fact that now the admissible $k$ varies in a wider range.
\end{oss}


\begin{open}
\begin{itemize} Below we present a list of open problems and work in progress.

    \item Generalize the rigidity results in the anisotropic setting, starting from the comparison proved in \cite{San2}.
    \item Generalize the rigidity results to other problems, such as the ones investigated in \cite{ACNT}, \cite{nunzia2022sharp}.
\end{itemize}
\end{open}


\section*{Acknowledgements}
 The authors Alba Lia Masiello and Gloria Paoli are supported by GNAMPA of INdAM. 
The author  Gloria Paoli is supported by the Alexander von Humboldt Foundation with an Alexander von Humboldt research fellowship.

\addcontentsline{toc}{chapter}{Bibliografia}
\bibliographystyle{plain}
\bibliography{biblio}

\begin{thebibliography}{10}

\bibitem{ACNT}
A.~Alvino, F.~Chiacchio, C.~Nitsch, and C.~Trombetti.
\newblock Sharp estimates for solutions to elliptic problems with mixed
  boundary conditions.
\newblock {\em J. Math. Pures Appl.}, 152:251—261, 2021.

\bibitem{AFLT}
A.~Alvino, V.~Ferone, G.~Trombetti, and P.-L. Lions.
\newblock Convex symmetrization and applications.
\newblock {\em Ann. Inst. H. Poincar\'{e} C Anal. Non Lin\'{e}aire},
  14(2):275--293, 1997.

\bibitem{lions_remark}
A.~Alvino, P.-L. Lions, and G.~Trombetti.
\newblock A remark on comparison results via symmetrization.
\newblock {\em Proc. Roy. Soc. Edinburgh Sect. A}, 102(1-2):37--48, 1986.

\bibitem{ALT}
A.~Alvino, P.-L. Lions, and G.~Trombetti.
\newblock Comparison results for elliptic and parabolic equations via {S}chwarz
  symmetrization.
\newblock {\em Ann. Inst. H. Poincar\'{e} Anal. Non Lin\'{e}aire}, 7(2):37--65,
  1990.

\bibitem{ANT}
A.~Alvino, C.~Nitsch, and C.~Trombetti.
\newblock A {T}alenti comparison result for solutions to elliptic problems with
  {R}obin boundary conditions.
\newblock {\em to appear on Comm. Pure Appl. Math.}

\bibitem{AGM}
V.~Amato, A.~Gentile, and A.~L. Masiello.
\newblock Comparison results for solutions to {$p$}-{L}aplace equations with
  {R}obin boundary conditions.
\newblock {\em Ann. Mat. Pura Appl. (4)}, 201(3):1189--1212, 2022.

\bibitem{ambrosio2000functions}
L.~Ambrosio, N.~Fusco, and D.~Pallara.
\newblock {\em Functions of bounded variation and free discontinuity problems}.
\newblock Oxford Mathematical Monographs. The Clarendon Press, Oxford
  University Press, New York, 2000.

\bibitem{AroTal}
G.~Aronsson and G.~Talenti.
\newblock Estimating the integral of a function in terms of a distribution
  function of its gradient.
\newblock {\em Boll. Un. Mat. Ital. B (5)}, 18(3):885--894, 1981.

\bibitem{AB}
M.~S. Ashbaugh and R.~D. Benguria.
\newblock On {R}ayleigh's conjecture for the clamped plate and its
  generalization to three dimensions.
\newblock In {\em Differential equations and mathematical physics
  ({B}irmingham, {AL}, 1994)}, pages 17--27. Int. Press, Boston, MA, 1995.

\bibitem{Brothers1988}
J.~E. Brothers and W.~P. Ziemer.
\newblock Minimal rearrangements of {S}obolev functions.
\newblock {\em J. Reine Angew. Math.}, 384:153--179, 1988.

\bibitem{burago}
Y.~D. Burago and V.~A. Zalgaller.
\newblock {\em Geometric inequalities}, volume 285 of {\em Grundlehren der
  Mathematischen Wissenschaften [Fundamental Principles of Mathematical
  Sciences]}.
\newblock Springer-Verlag, Berlin, 1988.
\newblock Translated from the Russian by A. B. Sosinski\u{\i}, Springer Series
  in Soviet Mathematics.

\bibitem{chavel}
I.~Chavel.
\newblock {\em Isoperimetric inequalities}, volume 145 of {\em Cambridge Tracts
  in Mathematics}.
\newblock Cambridge University Press, Cambridge, 2001.
\newblock Differential geometric and analytic perspectives.

\bibitem{nunzia2022sharp}
F.~Chiacchio, N.~Gavitone, C.~Nitsch, and C.~Trombetti.
\newblock Sharp estimates for the gaussian torsional rigidity with {R}obin
  boundary conditions.
\newblock {\em Potential Analysis}, pages 1--10, 2022.

\bibitem{fusco}
A.~Cianchi and N.~Fusco.
\newblock Steiner symmetric extremals in {P}\'{o}lya-{S}zeg\"{o} type
  inequalities.
\newblock {\em Adv. Math.}, 203(2):673--728, 2006.

\bibitem{degiorgi}
E.~De~Giorgi.
\newblock Sulla propriet\`a isoperimetrica dell'ipersfera, nella classe degli
  insiemi aventi frontiera orientata di misura finita.
\newblock {\em Atti Accad. Naz. Lincei Mem. Cl. Sci. Fis. Mat. Natur. Sez. Ia
  (8)}, 5:33--44, 1958.

\bibitem{FerVol}
A.~Ferone and R.~Volpicelli.
\newblock Minimal rearrangements of {S}obolev functions: a new proof.
\newblock {\em Ann. Inst. H. Poincar\'{e} C Anal. Non Lin\'{e}aire},
  20(2):333--339, 2003.

\bibitem{posteraro}
V.~Ferone and M.~R. Posteraro.
\newblock A remark on a comparison theorem.
\newblock {\em Comm. Partial Differential Equations}, 16(8-9):1255--1262, 1991.

\bibitem{kesavan}
S.~Kesavan.
\newblock On a comparison theorem via symmetrisation.
\newblock {\em Proc. Roy. Soc. Edinburgh Sect. A}, 119(1-2):159--167, 1991.

\bibitem{kes}
S.~Kesavan.
\newblock {\em Symmetrization \& applications}, volume~3 of {\em Series in
  Analysis}.
\newblock World Scientific Publishing Co. Pte. Ltd., Hackensack, NJ, 2006.

\bibitem{linp}
P.~Lindqvist.
\newblock On the definition and properties of {$p$}-superharmonic functions.
\newblock {\em J. Reine Angew. Math.}, 365:67--79, 1986.

\bibitem{NOI}
A.~L. Masiello and G.~Paoli.
\newblock A rigidity result for the robin torsion problem.
\newblock {\em arXiv preprint arXiv:2209.06706}, 2022.

\bibitem{ossy}
R.~Osserman.
\newblock The isoperimetric inequality.
\newblock {\em Bull. Amer. Math. Soc.}, 84(6):1182--1238, 1978.

\bibitem{pol}
G.~P\'{o}lya and G.~Szeg\"{o}.
\newblock {\em Isoperimetric {I}nequalities in {M}athematical {P}hysics}.
\newblock Annals of Mathematics Studies, No. 27. Princeton University Press,
  Princeton, N. J., 1951.

\bibitem{San2}
R.~Sannipoli.
\newblock Comparison results for solutions to the anisotropic {L}aplacian with
  {R}obin boundary conditions.
\newblock {\em Nonlinear Anal.}, 214:Paper No. 112615, 21, 2022.

\bibitem{T}
G.~Talenti.
\newblock Elliptic equations and rearrangements.
\newblock {\em Ann. Scuola Norm. Sup. Pisa Cl. Sci. (4)}, 3(4):697--718, 1976.

\bibitem{T2}
G.~Talenti.
\newblock Nonlinear elliptic equations, rearrangements of functions and
  {O}rlicz spaces.
\newblock {\em Ann. Mat. Pura Appl. (4)}, 120:160--184, 1979.

\bibitem{talenti}
G.~Talenti.
\newblock The standard isoperimetric theorem.
\newblock In {\em Handbook of convex geometry, {V}ol. {A}, {B}}, pages 73--123.
  North-Holland, Amsterdam, 1993.

\bibitem{T3}
G.~Talenti.
\newblock Inequalities in rearrangement invariant function spaces.
\newblock In {\em Nonlinear analysis, function spaces and applications, {V}ol.
  5 ({P}rague, 1994)}, pages 177--230. Prometheus, Prague, 1994.

\bibitem{V}
J.~L. V\'{a}zquez.
\newblock A strong maximum principle for some quasilinear elliptic equations.
\newblock {\em Appl. Math. Optim.}, 12(3):191--202, 1984.

\end{thebibliography}

\end{document}